\documentclass{amsart}
\usepackage{amsfonts}
\usepackage{amsmath,amscd}
\usepackage{amssymb}
\usepackage{amsthm}
\usepackage{newlfont}
\newcommand{\f}{\frac}

\newcommand{\ds}{\displaystyle}

 \newtheorem{thm}{Theorem}[section]
 \newtheorem{cor}[thm]{Corollary}
 \newtheorem{lem}[thm]{Lemma}
 \newtheorem{prop}[thm]{Proposition}
 \theoremstyle{definition}
 
 \theoremstyle{remark}

 \numberwithin{equation}{section}

\begin{document}

\title[On the tensor square of non-abelian nilpotent finite dimensional Lie algebras]
 {On the tensor square of non-abelian nilpotent finite dimensional Lie algebras}

\author[P. Niroomand]{Peyman Niroomand}
\address{School of Mathematics and Computer Science\\
Damghan University of Basic Sciences, Damghan, Iran}
\email{niroomand@dubs.ac.ir}






\begin{abstract}
For every finite $p$-group $G$ of order $p^n$ with derived subgroup
of order $p^m$, Rocco in \cite{roc} proved that the order of tensor
square of $G$ is at most $p^{n(n-m)}$. This upper bound has been
improved recently by author in \cite{ni}. The aim of the present
paper is to obtain a similar result for a non-abelian nilpotent Lie
algebra of finite dimension. More precisely, for any given
$n$-dimensional non-abelian nilpotent Lie algebra $L$ with derived
subalgebra of dimension $m$ we have $\mathrm{dim} (L\otimes L)\leq
(n-m)(n-1)+2$. Furthermore for $m=1$, the explicit structure of $L$
is given when the equality holds.
\end{abstract}

\maketitle
\section{Introduction}
The theory of tensor product of groups has been developed for Lie
algebras by Ellis in \cite{el1,el2}. A question naturally can be
asked here: which results on tensor square of groups, holds for Lie
algebras? every partial answer to this question may be an
interesting result in the theory of Lie algebras. Several authors
have been obtained in \cite{es3, es2, el1, el3, el2, es4, mo, sa}
similar argument to known result in group theory. In this article we
focus on a result of Rocco \cite{roc} which states for every finite
$p$-group of order $p^n$ with derived subgroup of order $p^m$, the
order of its tensor square is at most $p^{n(n-m)}$. Salemkar et. al.
\cite{sa} proved the same statement for Lie algebras of finite
dimension. Recently the author \cite{ni}, improved the upper bound
of Rocco and proved the order of tensor square is at most
$p^{(n-1)(n-m)+2}$. Furthermore when $m=1$ the structure of groups
under which the equality holds is completely described. The aim of
this paper is to show the similar inequality does hold for nilpotent
Lie algebras of finite dimension. As easily seen the ratio of two
mentioned upper bounds is $p^{n-(m+2)}$ which is a considerable
quantity so the result of this paper sharpens the inequality in
\cite{sa} a considerable amount.

\section{Tensor product and some known results}
 In this section first we introduce notation and terminology which are used in this paper.
 Then we summarize some known results without proof
 which will be used throughout the paper.

 Throughout  this paper let $F$ be a fixed field and $[,]$ denotes the Lie bracket.
For any two arbitrary Lie Algebras $M$ and $N$, an action of $M$ on
$N$ means a $F$-bilinear map $M\times N \longrightarrow N$ sending
$(m,n)$ to $^mn$ satisfying
\[^{[m,m']}n=^m(^{m'}n)-^{m'}(^mn)~\text{and}~ ^m[n,n']=[^mn,n']+[n,^mn'],\]
for all $m,m' \in M$ and $n,n' \in N$.

Lie multiplication in a Lie algebra $M$ can induce an action on
itself via $^{m'}m=[m',m]$ for all $m,m'\in M$. Also for any Lie
algebra $L$ a bilinear function $h:M\times N\longrightarrow L$ is
called a Lie pairing if for all $m,m' \in M$ and $n,n' \in N$
provided that
 \[\begin{array}{lcl}h([m, m'],n) = h(m, ^{m'}n) - h(m', ^mn),&&\vspace{.3cm}\\
h(m, [n, n']) = h(^{n'}m, n) - h(^nm, n'),&&\vspace{.3cm}\\
h(^nm,^{m'}n') = -[h(m,n),h(m',n')].\end{array}\]

The tensor product $M\otimes N$ is the Lie algebra generated by the
symbols $m\otimes n ~(m\in M, n\in N)$ subject to the following
relations:
\begin{itemize}
\item[(i)] $c(m\otimes n) = cm\otimes n = m\otimes cn,$
\item[(ii)] $(m+m')\otimes n = m\otimes n + m'\otimes n,$
$m\otimes(n + n') = m\otimes n + m\otimes n',$
\item[(iii)] $[m, m']\otimes n = m\otimes(^{m'}n)-m'\otimes(^mn),$
$m\otimes[n, n'] = (^{n'}m)\otimes n - (^nm)\otimes n',$
\item[(iv)] $[(m\otimes n), (m'\otimes n')] = -(n^m) \otimes(^{m'}n')$ for all $c\in F$, $m, m'\in M$ and $ n, n'\in N$.
\end{itemize}

Let $M$ and $N$ be two Lie algebras acting on each other. We are
able to find more results on the non-abelian tensor product
$M\otimes N$ by putting a condition on these actions. The actions
are called compatible if
\[^{(^nm)}{n'} = [n',^mn]~\text{and}~ ^{(^mn)}{m'}= [m',^nm] ~~\text{for all $m,m'\in M,n,n'\in N$}.\]
In the case, $M=N$ and all actions are given by Lie multiplication,
then $M\otimes M$ is called tensor square of $M$. The exterior
product $L\wedge L$ is obtained from $L\otimes L$ by imposing the
additional relation $l\otimes l$ for all $l$ in $L$ and the image of
$l\otimes l'$ is denoted by $l\wedge l'$ for all $l,l'\in L$ (see
\cite{el3,el1}).

By terminology of \cite{el2}, let $L\square L$ be the submodule of
$L\otimes L$ generated by the elements $l\otimes l$. One can check
that $L\square L$ lies in the centre of $L\otimes L$ and $L\wedge
L\cong L\otimes L/L\square L$.

Ellis in \cite{el3} showed that, $\mathcal{M}(L)$,  the Schur
multiplier of $L$, is isomorphic to the Kernel $(L\wedge
L\longrightarrow L^2, l\wedge l'\longrightarrow [l, l']).$  He also
in \cite{el2} generalized the universal quadratic functor $\Gamma$
(see \cite{wh} for more details). The quadratic functor $\Gamma$ let
us study the relation between the Lie exterior product and the Lie
tensor product.

 Now we mention the following consequence of propositions from \cite{el2,wh,sa} for the convenience of the reader.
 \begin{prop}$\mathrm{(See}$ \cite[Proposition 15, 16]{el2}\label{ta}$\mathrm{).}$ If $M$ and $N$ act trivially on each other, then
  there is an isomorphism
\[M\otimes N\cong  M\bigotimes_{mod}N, ~\text{where}~M\bigotimes_{mod}N\]
is the standard tensor product of $M$ and $N$.
 \end{prop}
 \begin{prop}$\mathrm{(See}$ \cite[Proposition 15, 16]{el2}\label{gam} $\mathrm{and}$ \cite{wh}$\mathrm{).}$
 \begin{itemize} \item[(i)]For any two Lie algebras  $M$ and $N$ there is an isomorphism
\[\Gamma(M\oplus N) = \Gamma(M)\oplus\Gamma(N)\oplus M\bigotimes_{mod} N.\]
\item[(ii)]There is an isomorphism  $\Gamma(F)\cong F$.
\end{itemize}
\end{prop}
\begin{prop}$\mathrm{(See}$ \cite[Remark 1.5 (d)]{sa}$\mathrm{and}$ \cite[Proposition 14, 15]{el2}$\mathrm{).}$\label{ps} Let $L$ be a Lie algebra, Then :
\begin{itemize}
\item[(i)]If $N\subseteq L^2\cap Z(G)$, then the sequence \[N\otimes L\longrightarrow L\otimes L\longrightarrow L/N\otimes L/N\longrightarrow 0\] is exact.
\item[(ii)] The sequence \[\Gamma(L/L^2)\stackrel{\psi}\longrightarrow L\otimes L \stackrel{\pi}\longrightarrow L\wedge L \longrightarrow 0\] is exact.
\item[(iii)] In particular, if  $L/L^2$ is a free Lie algebra, then the
sequence \[0\longrightarrow
\Gamma(L/L^2)\stackrel{\psi}\longrightarrow L\otimes L
\stackrel{\pi}\longrightarrow L\wedge L \longrightarrow 0\] is
exact.
\end{itemize}
\end{prop}
The following proposition provides the behavior of tensor product
with direct sums.
\begin{prop}$\mathrm{(See}$ \cite[Proposition 8]{el2}$\mathrm{).}$\label{ds}
There is an isomorphism
\[(M\oplus N)\otimes (M\oplus N)\cong (M\otimes M)\oplus (M\otimes N)\oplus (N\otimes M )\oplus (N\otimes N).\]
\end{prop}

Recall that a finite dimensional Lie algebra $L$ is called
Heisenberg if $L^2 = Z(L)$ and $\mathrm{dim}L^2=1$. Such algebras
are odd dimensional with basis $v_1,\ldots, v_{2m}, v$ and the only
non-zero multiplication between basis elements is $[v_{2i-1},
v_{2i}] = -[v_{2i}, v_{2i-1}] = v$ for $i = 1,\ldots m$. The symbol
$H(m)$ denotes the Heisenberg algebra of dimension $2m + 1$.
 The tensor square of Heisenberg algebra will be characterized in the next section by using the following result.

\begin{lem}$\mathrm{(See}$ \cite[Example 3]{es2} $\mathrm{and}$ \cite[Theorem 24]{mo}$\mathrm{).}$\label{h}
\begin{itemize}
\item[(i)]$\mathrm{dim}(\mathcal{M}(H(1)))=2$.
\item[(ii)]$\mathrm{dim}(\mathcal{M}(H(m)))=2m^2-m-1$ for all $m\geq 2$.
\end{itemize}
\end{lem}
Let $t(L)=\f{1}{2}n(n-1)-\mathrm{dim}\mathcal{M}(L)$ and $A(n)$
denoted abelian Lie algebra of dimension $n$. In \cite{es3,es4} the
structure on $L$ has been characterized only by the size of
$\mathcal {M}(L)$ as follows.
\begin{thm}$\mathrm{(See}$ \cite[Lemmas 2, 3 and Theorems 3, 5]{es3} $\mathrm{and}$ \cite[Theorem 1]{es4}$\mathrm{).}$ \label{sm}
Let $L$ be an $n$-dimensional nilpotent Lie algebra. Then
\begin{itemize}
\item[(i)] $t(L)\geq0$;
\item[(ii)] $t(L)=0$ if and only if $L\cong A(n);$
\item[(iii)] $t(L)=1$ if and only if $L\cong H(1);$
\item[(iv)] $t(L)=2$ if and only if $L\cong H(1)\oplus A(1);$
\item[(v)]$t(L)=3$ if and only if $L\cong H(1)\oplus A(2).$
\end{itemize}
\end{thm}
\section{Main Results}
In this section at first we obtain the tensor square of Heisenberg
algebras. Then we show that any $n$-dimensional non-abelian
nilpotent Lie algebra  $L$ with derived subalgebra of dimension $m$
satisfies
\[\mathrm{dim}(L\otimes L)\leq (n-m)(n-1)+2.\] In addition  for $m=1$, we obtain the explicit structure of $L$ when the equality holds.

\begin{lem}\label{de}For any two abelian Lie algebras $M$ and $N$ there is an isomorphism
\[(M\oplus N)\square (M\oplus N)\cong M\square M\oplus N\square N\oplus M\otimes N.\]
\end{lem}
\begin{proof} Since Proposition $11$ given in \cite{br} carries over
to the case of Lie algebra (see Proposition \ref{ds}). Hence the
similar proof of group theory given in \cite[Remark 5]{roc2} carries
over to the case of Lie algebra.
\end{proof}
\begin{cor}\label{de1} Let $L$ be a finite dimension Lie algebra, then \[\Gamma(L/L^2)\cong L\square L\cong L/L^2\square L/L^2.\]
\end{cor}
\begin{proof} By invoking Proposition \ref{ps} $(ii)$, Im$\psi=L\square L$. On the other hand, Proposition
\ref{ps} $(iii)$ implies that $\Gamma (A(1))\cong A(1)\square
A(1)\cong A(1)$.
 The rest of proof is obtained by the fact that there is a natural epimorphism
 $L\square L\longrightarrow L/L^2\square L/L^2$, Proposition \ref{gam} $(i)$ and Lemma \ref{de}.
\end{proof}
\begin{prop}\label{ha} Let $H(m)$ be a Heisenberg Lie algebra, then
 \[H(m)\otimes H(m)\cong H(m)/{H(m)^2}\otimes H(m)/{H(m)}^2 ,\]
 when $m\geq 2$.
 In the case $m=1$, $H(1)\otimes H(1)$ is an abelian Lie algebra of
 dimension 6.
\end{prop}
\begin{proof}
First suppose that $m\geq2$. According to the Lemmas \ref{h},
\ref{de}  and Corollary \ref{de1},
 \[\begin{array}{lcl} \mathrm{dim}( H(m)\otimes H(m))&=&\mathrm{dim}(H(m)\square H(m))+\mathrm{dim} \mathcal{M}(H(m))+\mathrm{dim} {H(m)}^2\vspace{.3cm}\\
&=&\ds\f{1}{2}2m(2m+1)+2m^2-m-1+1\vspace{.3cm}\\&=&4m^2.\end{array}\]
On the other hand, $ \mathrm{dim} (H(m)/{H(m)^2}\otimes
H(m)/{H(m)}^2)=4m^2$. Hence the result obtained by the natural Lie
epimorphism $H(m)\otimes H(m)\longrightarrow H(m)/{H(m)^2}\otimes
H(m)/{H(m)}^2.$

 In the case $m=1$, by a similar fashion $H(1)\otimes H(1)$ is a Lie algebra of dimension 6.
 One can easily check that $H(1)\otimes H(1)$ is abelian, which completes the
 proof.
\end{proof}
\begin{thm}\label{mt} Let  $L$ be an $n$-dimensional non-abelian  nilpotent Lie algebra   with derived subalgebra of dimension $m$. Then
\[\mathrm{dim} (L\otimes L)\leq (n-m)(n-1)+2.\] In particular, for  $m=1$ the equality holds if and only if
$L\cong H(1)\oplus A(n-3)$.
\end{thm}
\begin{proof} First suppose that $m=1$, since $L^2\subseteq Z(L)$ by using Proposition \ref{ha}, we may assume that $L^2\subsetneq Z(L)$.
Thus $L^2$ has a complement $A$ in $Z(L)$. The same argument shows
that $Z(L)/L^2$ has a complement $H/L^2$ in $L/L^2$. One can check
that $H$ is a Heisenberg Lie algebra and $L=H\oplus A.$   Hence by
Proposition \ref{ds},
   \[\begin{array}{lcl}\mathrm{dim}((H\oplus A)\otimes (H\oplus A))=\mathrm{dim}(H\otimes H)+2 \mathrm{dim}(H\otimes A)+\mathrm{dim}(A\otimes A)
   \end{array}\]
   On the other hands, Proposition \ref{ta} implies that \[H\otimes A\cong H/H^2\bigotimes_{mod}A.\]
   Two cases may be considered respect to the $\mathrm{dim}H$.
   \\Case $I.$ If $H\cong H(1)$, then by invoking of Propositions \ref{ds} and \ref{ha}, we have
   \[\begin{array}{lcl}\mathrm{dim} ((H\oplus A)\otimes (H\oplus A))&=&6+4(n-3)+(n-3)^2\vspace{.3cm}\\&=&(n-1)^2+2.
   \end{array}\]
   \\Case $II.$ If $H\cong H(m)$, then
   \[\begin{array}{lcl}\mathrm{dim} ((H\oplus A)\otimes (H\oplus A))&=&4m^2+4m(n-2m-1)+(n-2m-1)^2\vspace{.3cm}\\&=&(n-1)^2.
   \end{array}\]
  Let $m\geq 2$ and consider an ideal $N$ contained in $Z(L)\cap L^2$, induction hypothesis and Proposition \ref{ps} $(i)$ deduce that
\[\mathrm{dim}(L\otimes L)\leq \mathrm{dim}(L/N\otimes L/N)+\mathrm{dim}(N\otimes L).\]
Now by virtue of Proposition \ref{ta}, \[\mathrm{dim} (N\otimes
L)=\mathrm{dim}(N\bigotimes_{mod}L/L^2).\] Hence
\[\mathrm{dim}(L\otimes L)\leq n-m+(n-m)(n-2)+2=(n-1)(n-m)+2,\] as required.
\end{proof}
\begin{prop}Let  $L$ be an $n$-dimensional non-abelian  nilpotent Lie algebra  with derived subalgebra of dimension $m$ $(n\geq 4, m\geq 2)$. Then
\[\mathrm{dim} (L\otimes L)\leq (n-m)(n-1)+1~\text{or}~\mathrm{dim} (L\otimes L)\leq (n-m)(n-1)+2< n(n-m).\]
\end{prop}
\begin{proof}
Let $\mathrm{dim} (L\otimes L)>(n-m)(n-1)+1$, thus by Theorem
\ref{mt}, we should have $\mathrm{dim} (L\otimes L)=(n-m)(n-1)+2$.
By contrary we assume that $\mathrm{dim} (L\otimes L)=n(n-m)$. Hence
$m+n=2$.

 First suppose that $n=4$. Since $m=2$, Theorem \ref{sm} implies that $\mathrm{dim}\mathcal{M}(L)\leq 2$ and hence
$\mathrm{dim}(L\otimes L)\leq 7$. Now let $n\geq 5$,  induction
hypothesis and Proposition \ref{ps} $(i)$ implies  that
\[\mathrm{dim} (L\otimes L)\leq (n-m)(n-1)+1,\] which is a contradiction hence the result obtained.
\end{proof}

The last proposition shows developed Rocco's bound for a nilpotent
Lie algebra is less than the bound in Theorem \ref{sm} except for
the case $L\cong H(1)$.

\end{document}